\documentclass[number]{elsarticle}
\usepackage[hidelinks, colorlinks=true]{hyperref}
\usepackage{amscd,amssymb,amsmath,amsthm}
\usepackage[all]{xy}
\usepackage{array}
\usepackage{etoolbox}
\usepackage{mathtools}
\usepackage[a4paper, total={6in, 9.5in}]{geometry}
\usepackage[utf8]{inputenc}
\usepackage{tikz-cd}
\usepackage{tcolorbox}
\usepackage{enumitem}
\usepackage{url}

\makeatletter
\def\ps@pprintTitle{%
	\let\@oddhead\@empty
	\let\@evenhead\@empty
	\def\@oddfoot{\reset@font\hfil\thepage\hfil}
	\let\@evenfoot\@oddfoot
	
}

\makeatother

\newdir{ >}{!/8pt/\dir{}*\dir{>}}

\newtheorem*{th25}{Theorem 2.5}

\newtheorem*{th27}{Theorem 2.7}
\newtheorem*{th32}{Theorem 3.2}
\newtheorem*{th43}{Theorem 4.3}
\newtheorem*{cr48}{Corollary 4.8}

\newtheorem*{q3}{Question 3}
\newtheorem{theorem}{Theorem}[section]

\newtheorem{lemma}[theorem]{Lemma}
\newtheorem{corollary}[theorem]{Corollary}
\newtheorem{prop}[theorem]{Proposition}

\newtheorem{que}{Question}

\newcommand{\mybinom}[3][0.8]{\scalebox{#1}{$\dbinom{#2}{#3}$}}

\newtheorem{Remark}{Remark}
\numberwithin{equation}{theorem}

\theoremstyle{definition}
\newtheorem{definition}[theorem]{Definition}

\newcommand\dela[1]{}

\usepackage[utf8]{inputenc}

\journal{}

\begin{document}
	
	\begin{frontmatter}
		
		\title{The maximum number of triangles in a graph and its applications to special $p$-groups}
		
		\author[IISER TVM]{T.N. Mavely}
		\ead{tonynixonmavely17@iisertvm.ac.in}
		\author[IISER TVM]{V.Z. Thomas\corref{cor1}}
		\address[IISER TVM]{School of Mathematics,  Indian Institute of Science Education and Research Thiruvananthapuram,\\695551
			Kerala, India.}
		\ead{vthomas@iisertvm.ac.in}
		\cortext[cor1]{Corresponding author. \emph{Phone number}: +91 8921458330}
		
		\begin{abstract}
			We give a sharp bound on the number of triangles in a graph with fixed number of edges. We also characterize graphs that achieve the maximum number of triangles. Using the upper bound on number of triangles, we prove that if $G$ is a special $p$-group of rank $2 \leq k \leq \binom{d}{2} $, then $|\mathcal{M}(G)| \leq p^{\frac{d(d+2k-1)}{2}-k- \binom{d}{3}+ \binom{r}{3} + \mybinom[.55]{ \binom{d}{2} -k- \binom{r}{2} }{2} }$, where $r$ is such that $\binom{r}{2}  \leq \binom{d}{2} -k < \binom{r+1}{2} $. We also prove that, if $G$ is a $p$-group $(p \neq 2,3)$ of class $c \geq 3$, then $|\mathcal{M}(G)| \leq p^{\frac{d(m-e)}{2}+(\delta-1)(n-m)-\max(0,\delta-2)-\max(1,\delta-3)}$ and if $G$ is of coclass $r$ with class $c \geq 3$, then $|\mathcal{M}(G)| \leq p^{\frac{r^2-r}{2}+kr}$. 
		\end{abstract}
		
		
	\end{frontmatter}

	\section{Introduction} 
	
	The aim of this article is twofold: one is to solve a problem in extremal graph theory and then use that solution to give an application to special $p$-groups. Now we describe this in more detail. Assume $\mathcal{G}$ to be a simple, loopless and undirected graph i.e., there is exactly one edge connecting two vertices $v$ and $w$ of $\mathcal{G}$, and there are no edges whose two endpoints are actually the same vertex. In \cite{Rivi2002}, Rivin considers the following question:
	
	\begin{que}\label{q1}
		Let $\mathcal{G}$ be a graph with $E$ edges. Let $T$ be the number of triangles of $\mathcal{G}$. Show that there exists a constant $C$ such that $T \leq CE^{3/2}$ for all $\mathcal{G}$.
	\end{que}
	
	Rivin \cite{Rivi2002} answers this question and gives the following bound:
	\begin{align} \label{oldrivinbound}
		T \leq \frac{V-2}{\sqrt{V(V-1)}}\frac{2^{1/2}}{3}E^{3/2},
	\end{align}
	
	$V$ being the number of vertices of the graph $\mathcal{G}$. Rivin \cite[Theorem 2]{Rivi2002} further proves that the complete graph $K_n$ on $n$ vertices has maximal number of triangles among all the graphs with the same number of edges. More explicitly, this also says that for graphs with number of edges being triangular i.e $\binom{n}{2} \ \forall \ n \geq 2$, the maximum number of triangles is $\binom{n}{3}$. Motivated by this, we ask the following more general question:
	
	\begin{que} \label{q2}
		Among all graphs $\mathcal{G}$  with $n$ edges, what is the maximum number of triangles in $\mathcal{G}$?
	\end{que}
	We answer Question $\ref{q2}$ with a sharp bound in the following theorem:
	\begin{th25}
		Let $\binom{r}{2} \leq n < \binom{r+1}{2}$ and set $t:=n- \binom{r}{2}$. Then the number of triangles for any graph $\mathcal{G}$ with $n$ edges is less than or equal to $ \binom{r}{3} + \binom{t}{2}$. Furthermore there exists graphs that achieve the bound for each $n$.
	\end{th25}	
	
	For a graph with $E=\binom{r}{2}+t \ (1 \leq t \leq r-1 )$ edges, the minimum number of vertices required is $r+1$. Noting that $\frac{V-2}{\sqrt{V(V-1)}}$ is an increasing function, the smallest bound that can be obtained from $\ref{oldrivinbound}$ for given number of edges is by taking $V=r+1$, giving $\frac{r-1}{\sqrt{(r+1)r}}\frac{2^{1/2}}{3}(\binom{r}{2}+t)^{3/2}$. Note that Theorem $\ref{trianglebound}$ improves the bound in \ref{oldrivinbound}. If $t=0$, $E$ is a triangular number and both bounds coincide. Rivin \cite{Rivi2002} also asks the following natural question:
	\begin{que} \label{q3}
		Is there a simple characterization of graphs with $n$ edges which are triangle maximal (for all n)?
	\end{que}
	
	As an application of Theorem $\ref{trianglebound}$, we can answer Question \ref{q3} in the next theorem which characterizes graphs with fixed number of edges and maximal number of triangles.
	\begin{th27}
		Let $ \binom{r}{2}  \leq n < \binom{r+1}{2} $ and $t:=n- \binom{r}{2}$. Let $\mathcal{G}$ be the graph with $n$ edges and maximum number of triangles.
		\begin{enumerate}
			\item If $t=0 $, then $\mathcal{G}$ is the complete graph $K_r$ \cite[Theorem 2]{Rivi2002}.
	
			\item If $t = 1$ and $\mathcal{G}$ is connected, then $\mathcal{G}$ is the graph obtained by attaching a vertex $v$ with an edge to the complete graph $K_r$, where the edge is incident to a vertex of $K_r$.
			
			\item If $t=1$ and $\mathcal{G}$ is disconnected, then $\mathcal{G}$ is the graph $K_2 \cup K_r$.
						
			\item If $t \neq 0,1$, then $\mathcal{G}$ is the graph obtained by attaching a vertex $v$ with $t$ edges to the complete graph $K_r$, where the $t$ edges are incident to $t$ vertices of $K_r$.
			
		\end{enumerate}
	\end{th27}
	
	The above results are of independent interest, but our aim is to apply them to obtain a bound on the size of the Schur multiplier of special $p$-groups.	Let $G$ be a finite $p$-group of order $p^n$. A finite $p$-group $G$ is called a special $p$-group of rank $k$ if $\gamma_2 G=Z(G)=\Phi(G)$ is elementary abelian of order $p^k$. Special $p$-groups of rank $1$ are called extraspecial. The Schur multiplier of extraspecial $p$-groups has been studied in \cite[Corollary 4.16, p.223]{BeyTap1982}.
	In \cite{Hatu2020}, Hatui has given upper bounds for the size of Schur multiplier of special $p$-groups of rank $2$. In \cite[Theorem 1.1]{Rai2018}, Rai has given an upper bound for the size of the Schur multiplier of special $p$-groups of maximum rank $\binom{d}{2}$ when $d \geq 3$ where $d=d(G)$. To obtain sharp bounds for the size of the Schur multiplier of special $p$-groups of ranks $k \ (2 < k < \binom{d}{2})$ is more challenging. We use a novel approach of finding the maximum number of triangles in a graph in computing the size of the Schur multiplier for these cases. As an application of Theorem $\ref{trianglebound}$, we obtain the following theorem:
	\begin{th32} 
		Let $G$ be a special $p$-group of order $p^n$ and rank $k$, $2 \leq k \leq \binom{d}{2} $, where $d=d(G)$ is the minimal number of generators of $G$. Then $|\mathcal{M}(G)| \leq p^{\frac{d(d+2k-1)}{2}-k- \binom{d}{3} + \binom{r}{3} + \mybinom[.55]{ \binom{d}{2} -k- \binom{r}{2} }{2} }$, where $r$ is such that $ \binom{r}{2} \leq \binom{d}{2} -k < \binom{r+1}{2}$.
	\end{th32}
	
	\begin{Remark}
		Through calculations using GAP \cite{GAP4}, we observe that there are special $p$-groups of ranks $2,3,4 $ and $6$ that achieve the bound in Theorem $\ref{specialbound}$. In particular the groups given in Table $\ref{table}$ achieve the bound for corresponding values of $k$ and $d$. 
		
		$$\textbf{Table 1}$$
		\begin{center} \label{table}
			\begin{tabular}{ |c|c|c|c|c| } 
				\hline
				GroupId & Order & d(G) & Rank & $|\mathcal{M}(G)|$\\ 
				\hline
				$37$ & $3^5$ & $3$ & $2$ & $3^6$ \\ 
				$122$ & $3^6$ & $3$ & $3$ & $3^8$ \\
				$6477$ & $3^7$ & $4$ & $3$ & $3^{12}$  \\ 
				$263726$ & $3^8$ & $4$ & $4$ & $3^{14}$  \\
				\hline
			\end{tabular}
		\end{center}
	\end{Remark}

	Considering the particular case $k=\binom{d}{2}$ in Theorem $\ref{specialbound}$, we obtain the bound given in \cite[Theorem 1.1]{Rai2018}. Note that in \cite[Theorem 1.4(b)]{Hatu2020}, one obtains that the size of the Schur multiplier of special $p$-groups $(p \text{ odd})$ of rank $2$ is less than or equal to $p^{\frac{1}{2}d(d-1)+3}$ and this bound is achieved in \cite[Theorem 1.4(d)]{Hatu2020}. Setting $k=2$ in Theorem $\ref{specialbound}$, we obtain this bound in \cite{Hatu2020}. In \cite[Theorem 2]{EllWie1999}, Ellis and Wiegold proved that $|\mathcal{M}(G)| \leq p^{\frac{d(n-k-e)}{2}+(\delta-1)k-\max(0, \delta-2)}$ where $d=d(G)$, $\delta =d(G/Z)$ and $p^e$ is the exponent of $G^{ab}$. Rai sharpens this bound in \cite[Theorem 1.2]{Rai2017} by proving that $|\mathcal{M}(G)| \leq p^{\frac{1}{2}(d-1)(n-k-(\alpha_1-\alpha_d))+(\delta-1)k-\max(0, \delta-2)}$, where $G^{ab}=C_{p^{\alpha_1}} \times C_{p^{\alpha_2}} \times \ldots \times C_{p^{\alpha_d}} \ (\alpha_1 \geq \alpha_2 \geq \ldots \geq \alpha_d)$. We improve these bounds for groups of nilpotency class greater than or equal to 3 when $p \neq 2,3$. In particular, we prove:
	
	\begin{th43} 
		Let $G$ be a finite $p$-group of order $p^n \ (p \neq 2, 3),$ and nilpotency class $c \geq 3$. Then $|\mathcal{M}(G)| \leq p^{\frac{d(n-k-e)}{2}+(\delta-1)k-\max(0,\delta-2)-\max(1,\delta-3)}$.
	\end{th43}
	
	The coclass of $G$ is defined as $r:=n-c$. As a corollary to \cite[Theorem 1.1]{Rai2017}, Rai \cite[Theorem 1.5]{Rai2017} proved that $|\mathcal{M}(G)| \leq p^{\frac{1}{2}(r^2-r)+kr+1}$  for a non-abelian group of order $p^n$, coclass $r$ and with derived subgroup of size $p^k$. Using Theorem \ref{nil3bound}, we improve this bound for $p$-groups $(p \neq 2,3)$ of coclass $r$ with nilpotency class greater than or equal to 3. In particular, we prove:
	\begin{cr48} 
		Let $G$ be a finite $p$-group $(p \neq 2,3)$ of coclass $r$ with nilpotency class greater than or equal to $3$. Then $|\mathcal{M}(G)| \leq p^{\frac{r^2-r}{2}+kr}$.
	\end{cr48}

	$\textbf{Outline of the paper:}$ In section 2, we address a problem in extremal graph theory, namely Question $\ref{q2}$ and Question $\ref{q3}$ and prove Theorem $\ref{trianglebound}$ and Theorem $\ref{maxgraphcharaterize}$ which answer these questions completely. In section 3, we use Theorem $\ref{trianglebound}$ to obtain a bound on the size of the Schur multiplier of special $p$-groups of rank $k \ (2 \leq k \leq \binom{d}{2})$ in Theorem $\ref{specialbound}$. Section 4 is devoted to proving Theorem $\ref{nil3bound}$ which gives a bound on the size of the Schur multiplier of groups of nilpotency class strictly greater than $2$.
	We then	state Corollary $\ref{raicoclassboundimproved}$, which 	gives a bound on the size of the Schur multiplier of groups of coclass $r$ with nilpotency class strictly greater than $2$.
	
	$\textbf{Notation}:$ We will use $\mathcal{G}$ to denote a graph. For a graph $\mathcal{G}$, $V_{\mathcal{G}}$ will denote the set of vertices and $E_{\mathcal{G}}$ will denote the set of edges. We will denote $K_n$ as the complete graph with $n$ vertices. We shall use $d(v)$ to denote the degree of a vertex $v$. Given graphs $\mathcal{G}$ and $\mathcal{H}$, we denote $\mathcal{G} \cup \mathcal{H}$ as the graph whose vertex set and edge set  are the disjoint unions, respectively, of the vertex sets and edge sets of $\mathcal{G}$ and $\mathcal{H}$. We use the following fairly standard notation: we will denote the commutator subgroup of $G$ as $\gamma_2G$; we will use $\gamma_iG$ to denote the $i^{th}$ term of the lower central series of $G$; we will denote the nilpotency class of $G$ by $c$; $d(G)$ will denote the cardinality of a minimal generating set for $G$; we will use $G^{ab}$ to denote the abelianization of $G$. The Schur multiplier of $G$ is denoted by $\mathcal{M}(G)$ and the coclass of $G$ will be denoted by $r$.

	\section{An upper bound on the number of triangles in a graph with $n$ edges}
	
	The main aim of this section is to give a sharp upper bound for Question \ref{q2} and answer Question \ref{q3}. In \cite{Rivi2002}, Rivin considers the Question $\ref{q1}$ and answers it with $\ref{oldrivinbound}$. Rivin \cite[Theorem 2]{Rivi2002} further proves that among all graphs with $\binom{n}{2}$ edges, the complete graph $K_n$ on $n$ vertices has maximal number of triangles. In particular, he proves:
	
	\begin{theorem}[I. Rivin, \cite{Rivi2002}]\label{rivintheorem}
		In a graph $\mathcal{G}$ with no more than $n(n-1)/2$ edges, each edge is contained, on the average, in no more than $n-2$ triangles. Equality holds only for the complete graph $K_n$
	\end{theorem}
	Rivin \cite{Rivi2002} also asks the following natural questions:
	\begin{q3}
		Is there a simple characterization of graphs with $n$ edges which are triangle maximal (for all n)?
	\end{q3}
	\begin{que} \label{q4}
		Consider all graphs with $E$ edges and $V$ vertices. Is there a way to characterize the one with the most triangles?
	\end{que}
	The following lemma will be useful in estimating the number of triangles in a graph:
	
	\begin{lemma} \label{choose2bound}
		Let $a,b,c,d,m \in \mathbb{N} \cup \{0\}$ such that $a+b=c+d=m$ and $c \geq a \geq b$, then $ \binom{c}{2} + \binom{d}{2}  \geq \binom{a}{2} + \binom{b}{2} $. Equality holds if and only if $c=a$.
	\end{lemma}

	\begin{proof}
		The inequality holds if and only if $ \binom{c}{2} - \binom{a}{2}  \geq \binom{b}{2} - \binom{d}{2} $. As $ \binom{i}{2} =\sum \limits _{j=1}^{i-1}j$, the lemma holds if and only if $\sum \limits _{j=a}^{c-1}j \geq \sum \limits _{j=d}^{b-1}j$. The number of summands  on both sides are equal as $c-1-(a-1)=c-a=b-d=b-1-(d-1)$. Then the inequality is true if $a \geq d$ which holds as $a-d=c-b \geq 0$. If $c=a$, the equality follows trivially. So suppose that $c>a$. As the number of summands on both sides are equal, it is enough to see that $a > d$ to show $\sum \limits _{j=a}^{c-1}j > \sum \limits _{j=d}^{b-1}j$. But $ a > d$ holds as $a-d=c-b$ and $c > a \geq b$. This completes the lemma.
	\end{proof}
	
	The next lemma will be needed to prove Theorem $\ref{maxgraphcharaterize}$:
	\begin{lemma} \label{connectedincomplete}
		Let $\mathcal{G}$ be a connected graph. If $\mathcal{G}$ is not complete, then there exists vertices $x,y,z$ such that there are edges $xy$,$yz$, but no edge $xz$.
	\end{lemma}
	
	\begin{proof}
		We prove the result by contradiction. Suppose there is a connected graph, $\mathcal{G}$, which is not complete and has no such desired vertices. Then if there are edges $ab$ and $bc$, there is an edge $ac$. As $\mathcal{G}$ is not complete, there exists vertices $u$ and $w$ such that $u$ and $w$ are not adjacent to each other. As $\mathcal{G}$ is connected, there exists a path from $u$ to $w$ with alternating sequence of vertices and edges given by $uv_1v_2v_3\ldots v_nw$. As edges $uv_1$ and $v_1v_2$ exist, there must be an edge $uv_2$. Inductively, $v_3$ and $u$ are adjacent and so on to obtain an edge $uv_n$. As the edge $v_nw$ exists, we obtain the edge $uw$ which is a contradiction.
	\end{proof}
	
	The following definition \cite[p.183]{GroYel2006} gives us a method to modify a graph by identifying two vertices. This will be useful in showing that certain graphs with $n$ edges do not have maximal number of triangles.
	
	\begin{definition} \label{contraction}
		Let $\mathcal{H}$ be a subgraph of a graph $\mathcal{G}$. The contraction of $\mathcal{H}$ to a vertex is the replacement of $\mathcal{H}$ by a single vertex $k$. Each edge that joined a vertex $v \in V_{\mathcal{G}}-V_{\mathcal{H}}$ to a vertex in $\mathcal{H}$ is replaced by an edge with endpoints $v$ and $k$.
	\end{definition}
	
	\begin{Remark}
		As a particular case of the above definition, we define contracting a pair of vertices $v_1$ and $v_2$ to give a vertex $v$ as the contraction of the subgraph with vertices $v_1$ and $v_2$ and the edge connecting them (if any) by the vertex $v$.
	\end{Remark}
	
	The next theorem provides a precise answer to Question $\ref{q2}$:
	
	\begin{theorem} \label{trianglebound}
		Let $\binom{r}{2} \leq n < \binom{r+1}{2}$ and set $t:=n- \binom{r}{2}$. Then the number of triangles for any graph $\mathcal{G}$ with $n$ edges is less than or equal to $ \binom{r}{3} + \binom{t}{2}$. Furthermore there exists graphs that achieve the bound for each $n$.
	\end{theorem}	
	\begin{proof}
		The proof proceeds by induction on $n$. The result holds trivially for a graph with $1$ edge. Thus, we assume that the result holds for all numbers strictly less than n. We claim that the graph $\mathcal{G}$ with $n$ edges has at least one vertex $v$ with degree $1 \leq d(v) \leq r-1$. If all vertices had degree greater than or equal to $r$, then $\mathcal{G}$ will have at least $r+1$ vertices each with degree at least $r$. Thus $n \geq \frac{(r+1)r}{2}$, which is a contradiction as $n< \binom{r+1}{2} $. The number of triangles in $\mathcal{G}$ is the sum of number of triangles containing $v$ and the number of triangles not containing $v$. The number of triangles containing $v$ is less than or equal to $\frac{d(v)(d(v)-1)}{2}$. We will obtain the number of triangles not containing  $v$ by applying the induction hypothesis. Let $\mathcal{G}'$ denote the graph obtained by removing the vertex $v$ and all edges adjacent to $v$. Then the number of triangles not containing $v$ is equal to the number of triangles in $\mathcal{G}'$. As the number of edges of $\mathcal{G}'$ is $n-d(v)$, we can apply the induction hypothesis on $\mathcal{G}'$. If $t \geq d(v)$, then $ \binom{r}{2} \leq n-d(v) < \binom{r+1}{2} $ and by the induction hypothesis, the number of triangles of $\mathcal{G}'$ is less than or equal to $ \binom{r}{3} + \binom{t-d(v)}{2} $. Therefore, the number of triangles in $\mathcal{G}$ is $ \binom{r}{3} + \binom{t-d(v)}{2} + \binom{d(v)}{2}  \leq \binom{r}{3} + \binom{t}{2} $ by Lemma \ref{choose2bound}. If $t < d(v)$, then $ \binom{r-1}{2} \leq n-d(v) < \binom{r}{2} $. By the induction hypothesis, the number of triangles in $\mathcal{G}'$ is less than or equal to $ \binom{r-1}{3} + \binom{r-1+t-d(v)}{2} $. Therefore, the number of triangles in $\mathcal{G}$ is $ \binom{r-1}{3} + \binom{r-1+t-d(v)}{2} + \binom{d(v)}{2}  \leq \binom{r-1}{3} + \binom{r-1}{2} + \binom{t}{2} = \binom{r}{3} + \binom{t}{2} $ where the inequality holds by the Lemma \ref{choose2bound}. This completes the induction. To see that the bound is achieved, consider the following graph: Let $K_r$ be a complete graph. Attach to this graph a vertex with $t$ edges to $t$ vertices of the $K_r$ graph. This is a graph with $ \binom{r}{2}+t=n$ edges.  The $K_r$ graph has $ \binom{r}{3} $ triangles. The extra vertex contributes a triangle for every pair of the $t$ vertices it is attached to. This gives $ \binom{t}{2} $ triangles more and thus the graph has at least $ \binom{r}{3} + \binom{t}{2} $ triangles.
	\end{proof}
	
	Now that we have obtained the bound on the number of triangles in a graph with fixed number of edges, we shift our attention to answering Question $\ref{q3}$. We require the following proposition.
	
	\begin{prop}\label{disconnectedmaximalgraph}
		Let $\mathcal{G}$ be a graph with $n$ edges and maximum number of triangles. If $\mathcal{G}$ is disconnected, then $n = \binom{r}{2} +1$. Moreover for graphs with $n= \binom{r}{2} + 1$ edges and maximal number of triangles, the graph can be connected or disconnected. 
	\end{prop}
	
	\begin{proof}
		First we prove that if $\mathcal{G}$ is disconnected with at least three components, then it is not a graph with $n$ edges and maximal number of triangles. Let $v_iw_i$ denote an edge belonging to the $i^{th}$  component with $v_i$ and $w_i$ denoting the corresponding vertices. We contract vertices $v_1$ and $w_3$  to give the vertex $u_1$, contract vertices $v_2$ and $w_1$ to give the vertex $u_2$ and contract vertices $v_3$ and $w_2$ to give the vertex $u_3$ and let all other edges and vertices remain the same.  Then $u_1u_2u_3$ is a triangle present in the new graph that was not present in the previous graph. Note that all the other triangles are still retained. Thus we have produced a new graph with more number of triangles with the same number of edges which contradicts the maximality of triangles in the graph and hence $\mathcal{G}$ is a graph with at most two components with number of edges $a$ and $b$ respectively. If either component is not a complete graph, then it has three vertices $v_1,v_2,v_3$ such that there are edges $v_1v_2$ and $v_1v_3$, but no edge $v_2v_3$ by Lemma \ref{connectedincomplete}. Consider the edge $w_1w_2$ in the other component. Contracting vertices $v_2$ and $w_1$ to give the vertex $u_1$ and contracting vertices $v_3$ and $w_2$ to give the vertex $u_2$ forms a new triangle $v_1u_1u_2$ that was not in the previous graph contradicting the maximality of graph. Hence we may now assume that both the components are complete graphs. Thus the number of edges of both components are triangular numbers. If both components have more than one edge, then the components are given as $K_m$ and $K_n$ where $2 < m \leq n$. Remove a vertex $v$ and the $m-1$ edges adjacent to it in $K_m$ to get a $K_{m-1}$ graph and attach this vertex to $K_n$ in the following manner: $v$ is connected to $m-1$ vertices in $K_n$ by the $m-1$ edges. Note that now we have a new graph with the same number of edges. Observe that the number of triangles remain the same. To see this, note that the number of triangles lost from $K_m$ is $ \binom{m}{3} - \binom{m-1}{3} = \binom{m-1}{2}$ and the number of triangles added to $K_n$ is also $ \binom{m-1}{2} $ as $v$ has degree $m-1$ and the set of vertices adjacent to it form a complete graph. Now removing an edge from $K_{m-1}$ reduces the number of triangles by $m-3$ by Theorem \ref{rivintheorem}. Note that as $n>m-1$, there is a vertex in $K_n$ such that $v$ is not adjacent to it. The removed edge is attached to $v$ and this vertex in $K_n$ and this gives $m-1$ triangles as the attached edge can form a triangle with each of the previously added $m-1$ edges. Thus now we have a graph with the same number of edges but more number of triangles which contradicts the maximality of the graph we started with. So, if $\mathcal{G}$ is the graph with $n=a+b$ edges and maximal number of triangles with components with number of edges $a$ and $b$, then both $a$ and $b$ must be triangular numbers and one of them must be $1$. This completes the proof. Note that $K_2 \cup K_r$ and the complete $K_r$ graph with an extra vertex attached to one of the $r$ vertices both have $ \binom{r}{3} $ triangles. This shows that there are connected and disconnected graphs with $n= \binom{r}{2} +1$ edges and maximal number of triangles.
	\end{proof}
	
	With this, we now have all the ingredients to answer Question $\ref{q3}$. We prove:
	
	\begin{theorem} \label{maxgraphcharaterize}
		Let $ \binom{r}{2}  \leq n < \binom{r+1}{2} $ and $t:=n- \binom{r}{2}$. Let $\mathcal{G}$ be the graph with $n$ edges and maximum number of triangles.
		\begin{enumerate}
			\item If $t=0 $, then $\mathcal{G}$ is the complete graph $K_r$ \cite[Theorem 2]{Rivi2002}
			
			\item If $t=1$ and $\mathcal{G}$ is connected, then $\mathcal{G}$ is the graph obtained by attaching a vertex $v$ to a vertex of the complete graph $K_r$ with $1$ edge.
			
			\item If $t=1$ and $\mathcal{G}$ is disconnected, then $\mathcal{G}$ is the graph $K_2 \cup K_r$.
			
			\item If $t \neq 0,1$, then $\mathcal{G}$ is the graph obtained by attaching a vertex $v$ with $t$ edges to the complete graph $K_r$, where the $t$ edges incident to $t$ vertices of $K_r$.
			
		\end{enumerate}
	\end{theorem}
	\begin{proof}
		The proof proceeds by induction. The base case $n=1$ is trivial. So assume that the result is true for all numbers strictly less than $n$. Let $\mathcal{G}$ be a graph with $n$ edges and maximal number of triangles. If $n= \binom{r}{2} $, then $\mathcal{G}$ is the complete $K_r$ graph by Theorem \ref{rivintheorem}. Thus we can assume that $n$ is not a triangular number and therefore $t \neq 0$. As in Theorem \ref{trianglebound}, there exists a vertex $w$ with degree $1 \leq d(w) \leq r-1$ for a graph with $n$ edges. The number of triangles in $\mathcal{G}$ is the sum of the number of triangles containing $w$ and the number of triangles not containing $w$. Let $\mathcal{G}'$ denote the graph obtained by removing the vertex $w$ and all edges adjacent to $w$. Then the number of triangles not containing $w$ is equal to the number of triangles in $\mathcal{G}'$. 
		
		\textbf{Case 1:} $d(w)=t \neq 1$.
		
		The number of triangles containing $w$ is at most $ \binom{t}{2} $. As $\mathcal{G}'$ is a graph with $n-d(w)={r \choose 2}$ edges, the number of triangles in $\mathcal{G}'$ is at most ${r \choose 3}$ by Theorem \ref{trianglebound}. Thus, the number of triangles in $\mathcal{G}$ is at most ${r \choose 3}+{t \choose 2}$. Since $\mathcal{G}$ is a graph with maximum number of traingles and the upper bound in Theorem \ref{trianglebound} is attained, the number of triangles containing $w$ is ${t \choose 2}$ and the number of triangles in $\mathcal{G}'$ is ${r \choose 3}$. Thus the graph $\mathcal{G}'$ is a complete $K_r$ graph by Theorem \ref{rivintheorem}. Since the number of triangles containing $w$ is ${t \choose 2}$, every pair of vertices adjacent to $w$ must be connected by an edge. These edges belong to the $K_r$ graph because if it does not belong to the $K_r$ graph, the number of edges in $\mathcal{G}$ will exceed $n$. Thus $w$ is connected to $t$ vertices of the $K_r$ graph. This gives the graph described in part $4$.
		
		\textbf{Case 2:} $d(w)=t = 1$.

		The number of triangles containing $w$ is $0$. As $\mathcal{G}'$ is a graph with $n-d(w)={r \choose 2}$ edges, the number of triangles in $\mathcal{G}'$ is at most ${r \choose 3}$ by Theorem \ref{trianglebound}. Thus, the number of triangles in $\mathcal{G}$ is at most ${r \choose 3}$. Since $\mathcal{G}$ is a graph with maximum number of traingles and the upper bound in Theorem \ref{trianglebound} is attained, the number of triangles in $\mathcal{G}'$ is ${r \choose 3}$. Thus the graph $\mathcal{G}'$ is a complete $K_r$ graph by Theorem \ref{rivintheorem}. Now $w$ can be connected to the $K_r$ graph or disconnected from it. If the edge is connected to a vertex of the $K_r$ graph then we get the graph described in part $2$. If $w$ is disconnected from the $K_r$ graph, $w$ must  belong to a $K_2$ graph for it to have an edge adjacent to it. Then we get the graph $K_r \cup K_2$ described in part $3$.
		
		\textbf{Case 3:} $d(w) < t$.
		
		The number of triangles containing $w$ is at most ${d(w) \choose 2}$. As $\mathcal{G}'$ is a graph with $n-d(w)={r \choose 2}+t-d(w)$ edges, the number of triangles in $\mathcal{G}'$ is at most ${r \choose 3}+{t-d(w) \choose 2}$ by Theorem \ref{trianglebound}. Thus, the number of triangles in $\mathcal{G}$ is at most ${r \choose 3}+{t-d(w) \choose 2}+{d(w) \choose 2} < {r \choose 3}+{t \choose 2}$ by Lemma \ref{choose2bound}. This is a contradiction as a graph with $n$ edges and maximal number of triangles must have ${r \choose 3}+{t \choose 2}$ triangles. 
		
		\textbf{Case 4:} $t < d(w) < r-1$.
		
		The number of triangles containing $w$ is at most ${d(w) \choose 2}$. As $\mathcal{G}'$ is a graph with $n-d(w)={r-1 \choose 2}+t-d(w)+r-1$ edges, the number of triangles in $\mathcal{G}'$ is at most ${r-1 \choose 3}+{t-d(w)+r-1 \choose 2}$ by Theorem \ref{trianglebound}. Thus, the number of triangles in $\mathcal{G}$ is at most ${r-1 \choose 3}+{t-d(w)+r-1 \choose 2}+{d(w) \choose 2} < {r-1 \choose 3}+{r-1 \choose 2}+{t \choose 2}={r \choose 3}+{t \choose 2}$ where the inequality holds by the Lemma \ref{choose2bound}. This is a contradiction as a graph with $n$ edges and maximal number of triangles must have ${r \choose 3}+{t \choose 2}$ triangles. 
		
		\textbf{Case 5:} $t < d(w) = r-1$.
		
		The number of triangles containing $w$ is at most ${r-1 \choose 2}$. As $\mathcal{G}'$ is a graph with $n-d(w)={r-1 \choose 2}+t-d(w)+r-1={r-1 \choose 2}+t$ edges, the number of triangles in $\mathcal{G}'$ is at most ${r-1 \choose 3}+{t \choose 2}$. Thus, the number of triangles in $\mathcal{G}$ is at most ${r-1 \choose 3}+{t \choose 2}+{r-1 \choose 2} ={r \choose 3}+{t \choose 2}$. As $\mathcal{G}$ is the graph with maximum number of traingles and the upper bound in Theorem \ref{trianglebound} is attained, the number of triangles containing $w$ is ${r-1 \choose 2}$ and the number of triangles in $\mathcal{G}'$ is ${r-1 \choose 3}+{t \choose 2}$. Thus $\mathcal{G}'$ is a graph with maximal number of triangles. Since $\mathcal{G}'$ has ${r-1 \choose 2}+t$ edges, we can apply the induction hypothesis to $\mathcal{G}'$. If $\mathcal{G}'$ is connected, it is a complete $K_{r-1}$ graph with a vertex $v$ such that $v$ has $t$ edges adjacent to $t$ vertices of $K_{r-1}$. Since $w$ is contained in ${r-1 \choose 2}$ triangles, every pair of the $r-1$ vertices adjacent to $w$ must be connected by an edge and thus the vertices adjacent to $w$ form a complete $K_{r-1}$ graph. If $v$ belongs to the set of $r-1$ vertices adjacent to $w$, then $v$ has degree at least $r-2$ as it belongs to the $K_{r-1}$ graph. Including the edge connecting $v$ to $w$, $v$ has degree at least $r-1$. But according to the induction hypothesis $v$ has degree $t<r-1$ and therefore we have a contradiction. Thus $w$ is connected to all the vertices of the $K_{r-1}$ graph. This gives the graph described in part $2$ or the one described in part $4$. On the other hand, if $\mathcal{G}'$ is disconnected, then $\mathcal{G}'=K_{r-1} \cup K_2$. If $w$ is connected to vertices in $K_{r-1}$ and $K_2$, then those vertices must be adjacent for $w$ to have ${r-1 \choose 2}$ triangles. This is a contradiction as those vertices belong to two distinct components. If $w$ is adjacent to vertices entirely in $K_{r-1}$, then we obtain the graph $K_{r} \cup K_2$ described in part $3$. If $w$ is adjacent to $r-1$ vertices in $K_2$, then $r-1 \leq 2$. If $r-1=1$, then $r=2$, in which case $K_{r-1}=K_1$ and this component can be ignored, resulting in the graph described in part $2$. If $r-1=2$, then we have $w$ adjacent to vertices of a $K_2$ graph in $K_2 \cup K_2$ which gives $K_3 \cup K_2$ giving the graph described in part $3$. This completes the proof.
		
	\end{proof}
	
	\section{Bounds on order of Schur multiplier for special $p$-groups}
	A finite $p$-group $G$ is called a special $p$-group of rank $k$ if $\gamma_2G=Z(G)=\Phi(G)$ is elementary abelian of order $p^k$. Special $p$-groups of rank $1$ are called extraspecial. The Schur multiplier of a group $G$ is the second homology group $H_2(G,\mathbb{Z})$ of $G$, where the action of $G$ on $\mathbb{Z}$ is trivial. The author of \cite{Mill1952} proved that $H_2(G,\mathbb{Z})$ is isomorphic to $\mathcal{M}(G)$ where $\mathcal{M}(G)= \text{ker } (\kappa: G \wedge G \to \gamma_2G)$.  The Schur multiplier of extraspceial $p$-groups has been studied in \cite[Corollary 4.16, p223]{BeyTap1982}. In \cite{Hatu2020}, Hatui has given a complete classification of special $p$-groups of rank $2$ with respect to  the Schur multiplier along with upper bounds for the size of the Schur multiplier. In \cite[Theorem 1.1]{Rai2018}, Rai has given an upper bound for the size of the Schur multiplier of special $p$-groups of maximal possible rank $d \choose 2$ when $d \geq 3$ where $d=d(G)$. In this section we give bounds for the size of the Schur multiplier of special $p$-groups of all ranks $k$, $2 \leq k \leq {d \choose 2}$, which generalizes bounds given in \cite[Theorem 1.1]{Rai2018} and \cite{Hatu2020}. Consider the homomorphism defined in \cite[Proposition 1]{EllWie1999}:
	\begin{align*}
		\Psi_2 & : G^{ab} \otimes G^{ab} \otimes G^{ab} \to \gamma_2G/\gamma_3G \otimes G^{ab} \\
		\bar{x} \otimes \bar{y} \otimes \bar{z} & \mapsto \overline{[x,y]} \otimes \bar{z}+ \overline{[y,z]} \otimes \bar{x}+\overline{[z,x]} \otimes \bar{y} \\
	\end{align*}

	A more general version of the above homomorphism was considered in \cite{Elli1998} above Theorem 3. Note that for special $p$-groups, $\gamma_2G \otimes G^{ab}$ is the usual tensor product of vector spaces over $\mathbb{F}_p$. For a basis $B$ of $\gamma_2G$, we have the decomposition:
	$\gamma_2G \otimes G^{ab}=\bigoplus\limits_{[x_i,x_j] \in B} \langle [x_i,x_j] \rangle \otimes G^{ab}=\bigoplus\limits_{[x_i,x_j] \in B} \bigoplus\limits _{k=1}^d \langle [x_i,x_j] \rangle \otimes \langle x_k \rangle$. Let $P_{ijk}: \gamma_2G \otimes G \to \langle[x_i,x_j ] \rangle \otimes \langle x_k \rangle$  and $P_{ij}: \gamma_2G \otimes G \to \langle[x_i,x_j ] \rangle \otimes G^{ab}$  be the natural projection maps. The next proposition appears in \cite[Proposition 5(i)]{Elli1998} and will serve as our main tool.
	
	\begin{prop}\label{a2prop}[G. Ellis]
		Let $G$ be any $d$-generator group of order $p^n$. Set \[a_i=\text{dim}_{\mathbb{F}_p}(\text{Im}(\Psi_i)+\text{Im}(\mathfrak{s}_i)+\text{Im}(\mathfrak{t}_i))\] and set $a=a_2+a_3+\ldots+a_c$ where $\gamma_{c+1}G=1$. Then 
		\[|\mathcal{M}(G)||\gamma_2G| \leq p^{\frac{d(2n-d-1)}{2}-a}\]
	\end{prop}
	
	We will briefly describe the idea of the next theorem. Note that $a_i=0$ when $i > 2$ for $p$-groups of nilpotency class $2$. Then according to Proposition $\ref{a2prop}$, in order to obtain the bound on the size of the Schur multiplier of special $p$-groups, we try to find a lower bound for $a_2$ by finding  linearly independent elements in Im$(\Psi_2)$. To estimate the size of this set of linearly independent elements, we obtain an upper bound on the number of triangles in a graph with fixed number of edges using Theorem $\ref{trianglebound}$. With this, we come to the main theorem of this section:
	\begin{theorem}\label{specialbound}
		Let $G$ be a special $p$-group of order $p^n$ and rank $k$, $2 \leq k \leq {d \choose 2}$, where $d$ is the cardinality of the minimal generating set of $G$. Then $|\mathcal{M}(G)| \leq p^{\frac{d(d+2k-1)}{2}-k- \binom{d}{3}+ \binom{r}{3} + \mybinom[.55]{ \binom{d}{2} -k- \binom{r}{2} }{2} }$, where $r$ is such that ${r \choose 2} \leq {d \choose 2}-k < {r+1 \choose 2}$.
		
	\end{theorem}
	
	\begin{proof}
		Let $B$ be a basis of $\gamma_2G$ of size $k$ consisting of simple commutators, implying that order of $\gamma_2 G = p^k$. If  $\gamma_2 G=\Phi(G)$, then $n=d+k$ and the bound in Proposition \ref{a2prop} reduces to $|\mathcal{M}(G)| \leq p^{\frac{d(d+2k-1)}{2}-k-a}$. Now we will estimate $a$. In fact, we will give a lower bound of $\text{dim} _{\mathbb{F}_p}(\text{Im}(\Psi_2))$. Towards this end, we will exhibit a set $I$ of linearly independent elements in Im$(\Psi_2)$. 
		Consider
		\begin{align*}
			I:=\bigcup \limits _{1 \leq a < b < c \leq d} \{\Psi_2(\bar{x_a} \otimes \bar{x_b} \otimes \bar{x_c}) \mid [x_a,x_b] \in B \text{ or } [x_a,x_b] \notin B, [x_a,x_c] \in B \text{ or } [x_a,x_b], [x_a,x_c] \notin B, [x_b,x_c] \in B \} 
		\end{align*}
		We will prove that this set is linearly independent. To see this, suppose that 
		\begin{equation} \label{linearindependenceofPsi}
			\Psi_2(\bar{x_u} \otimes \bar{x_v} \otimes \bar{x_w})=\sum \limits_{\mathclap{\substack{\Psi_2(\bar{x_a} \otimes \bar{x_b} \otimes \bar{x_c}) \in I  \\ (a,b,c) \neq (u,v,w)}}}C_{abc}\Psi_2(\bar{x_a} \otimes \bar{x_{b}} \otimes \bar{x_{c}}),
		\end{equation}
		where $\Psi_2(\bar{x_u} \otimes \bar{x_v} \otimes \bar{x_w}) \in I$. If $[x_u,x_v] \in B $, then the map $P_{uvw}$, maps the LHS to $[x_u,x_v] \otimes \bar{x_w} \neq 0$ but maps the RHS to $0$. If  $[x_u,x_v] \notin B, [x_u,x_w] \in B $ then the map $P_{uwv}$, maps the LHS to $[x_u,x_w] \otimes \bar{x_v} \neq 0$ but maps the RHS to $0$. If $[x_u,x_v], [x_u,x_w] \notin B, [x_v,x_w] \in B$, then the projection map $P_{vwu}$ maps the LHS to $[x_v,x_w] \otimes \bar{x_u} \neq 0$, but it maps the RHS to 0. Thus, $I$ is linearly independent. To estimate the size of $I$, we  will define a set $\mathcal{I}$ which is in bijection with $I$. To do this, we first define $\mathcal{B}:=\{(i,j) \mid i<j,[x_i,x_j] \in B\}$. There is a natural bijective correspondence between $I$ and the set 	
		\begin{align*}
			\mathcal{I}:=\bigcup \limits _{1\leq i < j < k \leq d} \{(i,j,k) \mid (i,j) \in \mathcal{B} \text{ or } (i,j) \notin \mathcal{B}, (i,k) \in \mathcal{B} \text{ or } (i,j), (i,k) \notin \mathcal{B}, (j,k) \in \mathcal{B} \} 
		\end{align*}
		Considering the complement of $I$ in the set $\bigcup \limits _{1 \leq a < b < c \leq d} \{\Psi_2(\bar{x_a} \otimes \bar{x_b} \otimes \bar{x_c})\}$, note that $|I^c| = {d \choose 3}-|I|$. Then an upper bound on $I^c$ gives a lower bound on $I$. Observe that 
		\begin{align*}
			\mathcal{I}^c=\bigcup \limits _{1\leq i < j < k \leq d} \{(i,j,k) \mid (i,j),(i,k),(j,k) \notin \mathcal{B} \} 
		\end{align*}
		To estimate $\mathcal{I}^c$, we consider a graph with $d$ vertices numbered $1, \ldots, d$ and an edge connecting vertices $i$ and $j$ if $(i,j) \notin \mathcal{B}$. So the number of edges is ${d \choose 2}-k$. Observe that $|I^c| $ is the number of triangles in this graph. Hence we estimate the number of triangles of this graph. Suppose that ${r \choose 2} \leq {d \choose 2}-k \leq {r+1 \choose 2}$. Then by Theorem \ref{trianglebound}, the number of triangles in the graph is at most ${r \choose 3}+\mybinom[.55]{{d \choose 2}-k-{r \choose 2}}{2}$. Thus $I$ has a lower bound given by ${d \choose 3}-{r \choose 3}- \mybinom[.55]{{d \choose 2}-k-{r \choose 2}}{2}$ and hence the result.
		
	\end{proof}
	
	\section{Bound for Size of Schur multiplier of groups of nilpotency class greater than or equal to 3}
	Having considered special $p$-groups which are of nilpotency class 2, we now shift our attention to groups of nilpotency class strictly greater than 2. Let $G$ be a finite $d$-generated $p$-group of order $p^n$ with $|G^{ab}|=p^m$ and exponent of $G^{ab}$ is $p^e$. The next proposition will be crucially used for our results and was proved in \cite[Theorem 2]{EllWie1999}.
	\begin{prop} [G. Ellis, J. Wiegold, \cite{EllWie1999}]\label{Ellisbound}
		Let $G$ be a finite $p$-group with center $Z(G)$ and lower central series $1= \gamma_{c+1}G \unlhd \gamma_cG \unlhd \ldots \unlhd \gamma_1G =G$. Set $\bar{G}=G/Z(G)$ and consider the homomorphisms
		\[\Psi_2:\bar{G}^{ab} \otimes \bar{G}^{ab} \otimes \bar{G}^{ab} \to \gamma_2G/\gamma_3G \otimes \bar{G}^{ab}\]
		\[\bar{x} \otimes \bar{y} \otimes \bar{z} \mapsto \overline{[x,y]} \otimes \bar{z}+ \overline{[y,z]} \otimes \bar{x}+\overline{[z,x]} \otimes \bar{y}\]
		\[\Psi_3:\bar{G}^{ab} \otimes \bar{G}^{ab} \otimes \bar{G}^{ab} \otimes \bar{G}^{ab} \to \gamma_3G/\gamma_4G \otimes \bar{G}^{ab}\]
		\[\bar{x_1} \otimes \bar{x_2} \otimes \bar{x_3} \otimes \bar{x_4} \mapsto \overline{[[x_1,x_2],x_3]} \otimes \bar{x_4}+ \overline{[x_4,[x_1,x_2]]} \otimes \bar{x_3}+ \overline{[[x_3,x_4],x_1]} \otimes \bar{x_2}+ \overline{[x_2,[x_3,x_4]]} \otimes \bar{x_1}\]
		Here $\bar{x}$ denotes the image in $\bar{G}$ of the element $x \in G$, $\overline{[x,y]}$ denotes the image in $\gamma_2G/\gamma_3G$ of the commutator $[x,y] \in G$ and $\overline{[[x,y],z]}$ denotes the image in $\gamma_3G/\gamma_4G$ of the commutator $[[x,y],z] \in G$. Then $|\mathcal{M}(G)||\gamma_2G||\text{image }(\Psi_2)| \leq |M(G^{ab})||\gamma_2G/\gamma_3G \otimes \bar{G}^{ab}||\gamma_3G/\gamma_4G \otimes \bar{G}^{ab}|\ldots|\gamma_cG \otimes \bar{G}^{ab}|.$
		Moreover the next inequality appears as a remark below \cite[Theorem 2]{EllWie1999}:
		\[|\mathcal{M}(G)||\gamma_2G||\text{image }(\Psi_2)||\text{image }(\Psi_3)| \leq |M(G^{ab})||\gamma_2G/\gamma_3G \otimes \bar{G}^{ab}||\gamma_3G/\gamma_4G \otimes \bar{G}^{ab}|\ldots|\gamma_cG \otimes \bar{G}^{ab}|.\]
	\end{prop}
	Using this inequality, Ellis and Wiegold \cite[Theorem 2]{EllWie1999} prove that $|\mathcal{M}(G)| \leq p^{\frac{d(m-e)}{2}+(\delta-1)(n-m)-\max(0, \delta-2)}$. In \cite[Theorem 1.2]{Rai2017}, Rai sharpened the bound to give $|\mathcal{M}(G)| \leq p^{\frac{1}{2}(d-1)(n-k-(\alpha_1-\alpha_d))+(\delta-1)(n-m)-\max(0, \delta-2)}$.
	The next proposition appears in \cite[Lemma 2.2]{ShaNirJoh2020} and will be used in the proof.
	
	\begin{prop} \label{Niroo}
		Let $G$ be a two generator group of order $p^n$ ($p \neq 2$) and nilpotency class c. Then $\text{Im } \Psi_i \neq 0$ for all odd integers $i$ such that $3 \leq i \leq c$.
	\end{prop}
	
	The following theorem improves \cite[Theorem 2]{EllWie1999}:
	
	\begin{theorem} \label{nil3bound}
		Let $G$ be a finite $p$-group of order $p^n$ and nilpotency class $c \geq 3$ where $p \neq 2, 3$. Then $|\mathcal{M}(G)| \leq p^{\frac{d(m-e)}{2}+(\delta-1)(n-m)-\max(0,\delta-2)-\max(1,\delta-3)}$.
	\end{theorem}
	
	\begin{proof}
		Following \cite[Theorem 2]{EllWie1999}, it is enough to estimate $|\text{Im }(\Psi_3)|$. The generating set of $G/Z(G)$, $\{\bar{x_1},\bar{x_2}, \ldots, \bar{x_{\delta}}\}$, can be chosen in such a manner that $(G/Z(G))^{ab} \cong \langle\bar{x_1} \rangle \times \langle\bar{x_2}\rangle \times \ldots \times \langle\bar{x_{\delta}}\rangle$. Note that $\gamma_3G/\gamma_4G$ is non-trivial and let $\overline{[[y_1,y_2],y_3]}$ be an element of the minimal generating set and without loss of generality $y_i \in \{x_1,x_2,x_3\}$.		
		Assume $\delta > 3$ and consider the following $\delta-3$ elements
		\[\Psi_3(\bar{y_1} \otimes \bar{y_2} \otimes\bar{y_3} \otimes \bar{x_4}),\Psi_3(\bar{y_1} \otimes \bar{y_2} \otimes\bar{y_3} \otimes \bar{x_5}),\ldots, \Psi_3(\bar{y_1} \otimes \bar{y_2} \otimes\bar{y_3} \otimes \bar{x_{\delta}}).\]
		Set $A := \gamma_3G/\gamma_4G$ and note that \[A \otimes \bar{G}^{ab} \cong (A \otimes \langle \bar{x_1} \rangle) \oplus \ldots \oplus (A \otimes \langle \bar{x_{\delta}} \rangle ).\]
		Among the above list of elements, for $i > 3$, $\Psi_3(\bar{y_1} \otimes \bar{y_2} \otimes\bar{y_3} \otimes \bar{x_i})$ is the only element to have a non-trivial projection in $A \otimes \langle \bar{x_i} \rangle $. Thus these $\delta-3$ elements are linearly independent and we have \[|\text{Im }(\Psi_3)| \geq p^{\delta-3}.\]
		Suppose that $\delta = 3$. If $y_i$ in $\overline{[[y_1,y_2],y_3]}$ are not all distinct, there exists $x_i \notin \{y_1,y_2,y_3\}$ and then $\Psi_3(\bar{y_1} \otimes \bar{y_2} \otimes\bar{y_3} \otimes \bar{x_i})$ has a non-trivial projection in $A \otimes \langle \bar{x_i} \rangle $ and the statement holds. Thus we can assume that each commutator generating $\gamma_3/\gamma_4$ have distinct terms in the commutator and they are:
		\begin{align*}
			\{[[x_1,x_2], x_3], \ [[x_2,  x_3], x_1], \ [[x_3, x_1], x_2], \ [[x_1, x_3], x_2], \ [[x_2, x_1], x_3], \ [[x_3, x_2], x_1], \\ [x_1, [x_2, x_3]], \ [x_2, [x_1, x_3]], \ [x_3, [x_1, x_2]], \ [x_1, [x_3, x_2]], \ [x_2, [x_3, x_1]], \ [x_3, [x_2, x_1]]   \}
		\end{align*} 
		Using the relation $[[a,b],c] \equiv [c,[b,a]] \text{ mod }\gamma_3G/\gamma_4G$, one can note that the $ 12$ elements listed above can be generated by $\{[[x_1,x_2],x_3],[[x_2,x_3],x_1],[[x_3,x_1],x_2]\}$. Now using the Hall-Witt identity, we obtain $[[x_1,x_2],x_3]+[[x_2,x_3],x_1]+[[x_3,x_1],x_2]=0$. Without loss of generality, we observe that $\gamma_3G/\gamma_4G$ is generated by $\{[[x_1,x_2],x_3]\}$ or by $\{[[x_1,x_2],x_3],[[x_3,x_1],x_2]\}$. If $\gamma_3G/\gamma_4G$ is generated by $\{ [x_1,x_2,x_3] \} $, we will show that either $\Psi_3(\bar{x_1} \otimes \bar{x_2} \otimes\bar{x_3} \otimes \bar{x_1})$ or $\Psi_3(\bar{x_1} \otimes \bar{x_2} \otimes\bar{x_3} \otimes \bar{x_2})$ is non-trivial. If $\gamma_3G/\gamma_4G$ is generated by $\{[[x_1,x_2],x_3],[[x_3,x_1],x_2]\}$, we will show that $\Psi_3(\bar{x_1} \otimes \bar{x_2} \otimes\bar{x_3} \otimes \bar{x_1})$ is non-trivial. 
		Suppose $\gamma_3G/\gamma_4G$ is generated by $\{[[x_1,x_2],x_3],[[x_3,x_1],x_2]\}$, then the projection of $\Psi_3(\bar{x_1} \otimes \bar{x_2} \otimes\bar{x_3} \otimes \bar{x_1})$ in $A \otimes \langle \bar{x_1} \rangle$ is given by $[[x_1,x_2],x_3] \otimes x_1+[x_2,[x_3,x_1]] \otimes x_1=[[x_1,x_2],x_3] \otimes x_1-[[x_3,x_1],x_2]] \otimes x_1$. Now noting that if $\{a,b\}$ is a generating set of a group, then $\{a-b,a\}$ is also a generating set of that group,  we observe that $[[x_1,x_2],x_3]-[[x_3,x_1],x_2]]$ belongs to a set of minimal generators of $\gamma_3G/\gamma_4G$. Thus $[[x_1,x_2],x_3] \otimes x_1-[[x_3,x_1],x_2]] \otimes x_1$ is non-trivial. Now suppose $\gamma_3G/\gamma_4G$ is generated by $[[x_1,x_2],x_3]$. By Hall-Witt identity we obtain, $[[x_1,x_2],x_3]+[[x_2,x_3],x_1]+[[x_3,x_1],x_2]=0$ in $\gamma_3G/\gamma_4G$, which can be considered as the equation $1+m+n \equiv 0 \text{ mod }p$. From this either $1-m$ or $1-n$ is non-zero $\text{mod }p$ (as $p \neq 3$) and is thus a generator of the cyclic group $\gamma_3G/\gamma_4G$, i.e. $[[x_1,x_2],x_3]-[[x_3,x_1],x_2]$ or $[[x_1,x_2],x_3]-[[x_2,x_3],x_1]$ is a generator.  Now, if $[[x_1,x_2],x_3]-[[x_3,x_1],x_2]$ is a generator, then $\Psi_3(\bar{x_1} \otimes \bar{x_2} \otimes\bar{x_3} \otimes \bar{x_1})$ is non-trivial as its projection in $A \otimes \langle \bar{x_1} \rangle$ is $[[x_1,x_2],x_3] \otimes x_1+[x_2,[x_3,x_1]] \otimes x_1=([[x_1,x_2],x_3]-[[x_3,x_1],x_2]) \otimes x_1$ which is non-zero. Now, if $[[x_1,x_2],x_3]-[[x_2,x_3],x_1]$ is a generator, then $\Psi_3(\bar{x_1} \otimes \bar{x_2} \otimes\bar{x_3} \otimes \bar{x_2})$ is non-trivial  as its projection in $A \otimes \langle \bar{x_1} \rangle$ is $[[x_1,x_2],x_3] \otimes x_2+[[x_3,x_2],x_1]] \otimes x_2=([[x_1,x_2],x_3]-[[x_2,x_3],x_1]) \otimes x_2$ is non-zero. If $\delta=2$, then Proposition \ref{Niroo} gives a non-trivial element in Im $\Psi_3 $ which completes the theorem.
		
	\end{proof}
	
	In \cite[Theorem 1.1]{Rai2017}, the author showed that \cite[Theorem 2]{EllWie1999} implied the bound $|\mathcal{M}(G)| \leq p^{\frac{1}{2}(d-1)(n-k-2)+1}$. Using Theorem \ref{nil3bound}, we obtain the following improvement:
	
	\begin{corollary} \label{raisimplerbound}
		Let $G$ be a finite $p$-group of nilpotency class $c \geq 3$ of order $p^n \ (p \neq 2,3)$ with $|\gamma_2|=p^k$ and $d(G)=d$. Then $|\mathcal{M}(G)| \leq p^{\frac{1}{2}(d-1)(n-k-2)}$.
	\end{corollary}

	In \cite[Theorem 1.2]{Rai2017}, the author sharpened the bound in \cite[Theorem 2]{EllWie1999}. Using the same ideas, the bound in Theorem \ref{nil3bound} can be further sharpened as follows:
	
	\begin{corollary} \label{raisharpenednil3boundimproved}
		Let $G$ be a finite $p$-group of nilpotency class $c \geq 3$ of order $p^n \ (p \neq 2,3)$ with $d(G)=d$, $d(G/Z)=\delta$ and $G^{ab}=C_{p^{\alpha_1}} \times C_{p^{\alpha_2}} \times \ldots \times C_{p^{\alpha_d}} \ (\alpha_1 \geq \alpha_2 \geq \ldots \geq \alpha_d)$. Then $|\mathcal{M}(G)| \leq p^{\frac{1}{2}(d-1)(n-k-(\alpha_1-\alpha_d))+(\delta-1)k-\max(0,\delta-2)-\max(1,\delta-3)}$.
	\end{corollary}
	
	This easily leads to the following corollary:
	
	\begin{corollary} \label{rainothomocyclicboundimproved}
		Let $G$ be a finite $p$-group of nilpotency class $c \geq 3$ of order $p^n$ with $d(G)=d$, $\gamma_2(G)=p^k$ and $G^{ab}C_{p^{\alpha_1}} \times C_{p^{\alpha_2}} \times \ldots \times C_{p^{\alpha_d}} \ (\alpha_1 \geq \alpha_2 \geq \ldots \geq \alpha_d)$ where $p \neq 2,3$. Then \[|\mathcal{M}(G)| \leq p^{\frac{1}{2}(d-1)(n+k-2-(\alpha_1-\alpha_d))}.\]
		In particular, if $G^{ab}$ is not homocyclic, then
		\[|\mathcal{M}(G)| \leq p^{\frac{1}{2}(d-1)(n+k-3)}.\]
	\end{corollary}
	
	In \cite[Theorem 1.3]{Verm1974}, the author proved that if $K$ is a central subgroup of $G$ such that the restriction homomorphism from $\mathcal{M}(G)$ to $\mathcal{M}(K)$ is zero, then 
	\[|\gamma_2 (G)| |\mathcal{M}(G)| \leq |(G/K)^{ab} \otimes K|p^{\frac{1}{2}(m-r)(m+r-1)},\]
	where $|G/K|=p^m$ and $|\gamma_2 (G)K/K|=p^r$. In \cite[Theorem 1.4]{Rai2017}, the author improves the bound in \cite[Theorem 1.3]{Verm1974} for non-abelian finite $p$-groups of nilpotency class at least $3$ giving
	\[|\gamma_2 (G)| |\mathcal{M}(G)| \leq |(G/K)^{ab} \otimes K|p^{\frac{1}{2}(m-r)(m+r-2)+1}.\]
	This can be further improved for non-abelian finite $p$-groups of nilpotency class at least $4$:
	
	\begin{corollary} \label{raiVermaniboundimproved}
		Let $G$ be a finite $p$-group of nilpotency class at least $4$ and $K$ a central subgroup of $G$ such that the restriction homomorphism from $\mathcal{M}(G)$ to $\mathcal{M}(K)$ is zero. If $|G/K|=p^m$ and $|\gamma_2 (G) K/K|=p^r$, then 
		\[|\gamma_2 (G)| |\mathcal{M}(G)| \leq |(G/K)^{ab} \otimes K|p^{\frac{1}{2}d(G/K)(m+r-2)}.\]
		In particular,
		\[|\gamma_2 (G)| |\mathcal{M}(G)| \leq |(G/K)^{ab} \otimes K|p^{\frac{1}{2}(m-r)(m+r-2)}.\]
	\end{corollary}
	
	As a corollary to \cite[Theorem 1.1]{Rai2017}, Rai in \cite[Theorem 1.5]{Rai2017} proved that for non-abelian groups of order $p^n$ and coclass $r$, with derived subgroup of size $p^k$, $|\mathcal{M}(G)| \leq p^{\frac{1}{2}(r^2-r)+kr+1}$. Applying the same argument to the bound from Corollary \ref{raisimplerbound}, one immediately obtains the following improvement to \cite[Theorem 1.5]{Rai2017}:
	
	\begin{corollary}\label{raicoclassboundimproved}
		Let $G$ be a finite $p$-group of coclass $r$ with nilpotency class strictly greater than $2$ and $p \neq 2,3$. Denote the size of $\gamma_2G$ by $p^k$. Then $|\mathcal{M}(G)| \leq p^{\frac{r^2-r}{2}+kr}$. 
	\end{corollary}
	
	\section*{Acknowledgements} V. Z. Thomas acknowledges research support from SERB, DST, Government of India grant MTR/2020/000483. The authors asked Marcin Mazur whether he knows a bound for the set $Y_X:=\{(a,b,c) \mid (a,b),(a,c),(b,c) \in X, 1 \leq a < b < c \leq d\}$ where $X$ is a subset of $\{(a,b) \mid  1 \leq a < b \leq d\}$. Mazur informed us that there is a bijection between $Y_X$ and the triangles in a graph with vertices $\{1, \ldots , d\}$ and edges $X$ and pointed us to Rivin's paper \cite{Rivi2002}. We are grateful to him for this input as it simplified the exposition in Section 3. We also thank Robert F. Morse for providing us with the GAP code to compute the examples mentioned in the introduction. We thank Komma Patali for helpful discussions.

	\bibliographystyle{amsplain}
	\bibliography{MT}

\end{document}